\documentclass[10pt,bezier]{article}
\usepackage{amsmath,amssymb,amsfonts,graphicx,caption,subcaption}
\usepackage[colorlinks,linkcolor=red,citecolor=blue]{hyperref}

\newtheorem{prethm}{{\bf Theorem}}[section]
\newenvironment{thm}{\begin{prethm}{\hspace{-0.5
em}{\bf.}}}{\end{prethm}}
\newtheorem{prepro}{{\bf Theorem}}

\newtheorem{precor}[prethm]{{\bf Corollary}}
\newenvironment{cor}{\begin{precor}{\hspace{-0.5
em}{\bf.}}}{\end{precor}}
\newtheorem{preconj}[prethm]{{\bf Conjecture}}

\newtheorem{preremark}[prethm]{{\bf Remark}}
\newenvironment{remark}{\begin{preremark}\em{\hspace{-0.5
em}{\bf.}}}{\end{preremark}}
\newtheorem{prelem}[prethm]{{\bf Lemma}}
\newenvironment{lem}{\begin{prelem}{\hspace{-0.5
em}{\bf.}}}{\end{prelem}}
\newtheorem{preque}[prethm]{{\bf Question}}

\newtheorem{preobserv}[prethm]{{\bf Observation}}

\newtheorem{preproposition}[prethm]{{\bf Proposition}}

\newtheorem{preproof}{{\bf Proof.}}
\newtheorem{preprooff}{{\bf Proof}}

\newenvironment{proof}[1]{\begin{preproof}{\rm
#1}\hfill{$\Box$}}{\end{preproof}}

\newtheorem{preproofs}{{\bf Second proof of }}

\newtheorem{preprooft}{{\bf Third proof of }}

\newtheorem{preproofF}{{\bf Proof of}}

\title{\bf\Large 
A necessary and sufficient condition for the existence of $\{p,p+1,q-1,q\}$-orientations in simple graphs
}
\author{{\normalsize{\sc Morteza Hasanvand${}$} }\vspace{3mm}
\\{\footnotesize{${}$\it Department of Mathematical
 Sciences, Sharif
University of Technology, Tehran, Iran}}
{\footnotesize{}}\\{\footnotesize{ $\mathsf{morteza.hasanvand@alum.sharif.edu }$ }}}
\date{}
\def\p {p}
\def\q {q}
\begin{document}
\maketitle
\begin{abstract}{ 
Let $G$ be a simple graph and let $\p$ and $\q$ be two integer-valued functions on $V(G)$ with $p< q$ in which for each $v\in V(G)$, $q(v) \ge \frac{1}{2}d_G(v)$ and $\p(v) \ge \frac{1}{2} \q(v)-2$. In this note, we show that $G$ has an orientation such that for each vertex $v$, $d^+_G(v)\in\{\p(v),p(v)+1,q(v)-1,\q(v)\}$ if and only if it has an orientation such that for each vertex $v$, $p(v) \le d^+_G(v)\le \q(v)$ where $d^+_G(v)$ denotes the out-degree of $v$ in $G$. From this result, we refine a result due to Addario-Berry, Dalal, and Reed (2008) in bipartite simple graphs on the existence of degree constrained factors.
} {\small
}
\end{abstract}
%
%
%==============================================================================
%
%
%
%
%
%
%
%
%
%
%
%
%									Introduction
%==============================================================================
\section{Introduction}
%==============================================================================
%									Definitions
In this note, graphs have no loops, but multiple edges are allowed, and a simple graph have neither multiple edges nor loops.
Let $G$ be a graph. 
The vertex set and the edge set of $G$ are denoted by $V(G)$ and $E(G)$, respectively.
We denote by $d_G(v)$ the degree of a vertex $v$ in the graph $G$, whether $G$ is directed or not.
If $G$ has an orientation $D$, the out-degree and in-degree of $v$ are denoted by $d_D^+(v)$ and $d_D^-(v)$; when $D$ is clear from the context, we only write $d_G^+(v)$ and $d_G^-(v)$.
We denote by $G[A]$ the induced subgraph of $G$ with the vertex set $A$ containing
precisely those edges of $G$ whose ends lie in $A$.
Likewise, we denote by $D[A]$ the induced subdigraph of $D$ with the vertex set $A$ containing
precisely those edges of $D$ whose ends lie in $A$.
Let $L:V(G)\rightarrow 2^\mathbb{Z}$ be a function.
An orientation of $D$ of $G$ is said to be 
(i)
{\it $L$-orientation}, if for each vertex $v$, $d^+_D(v)\in L(v)$, (ii) {\it $(\p,\q)$-orientation}, if for each vertex $v$, $\p(v)\le d^+_D(v)\le \q(v)$,
where $\p$ and $\q$ are two integer-valued functions on $V(G)$.
Likewise, a factor $F$ of the graph $G$ is said to be {\it $L$-factor}, if for each vertex $v$, $d_F(v)\in L(v)$, (ii) {\it $(g,f)$-factor}, if for each vertex $v$, $g(v)\le d_F(v)\le f(v)$, where $g$ and $f$ are two integer-valued functions on $V(G)$.
%
%
%==============================================================================
%
%
%
%
%
%==============================================================================
%
%

In 1976 Frank and Gy{\'a}rf{\'a}s formulated the following a criterion for the existence of $(p,q)$-orientations which generalizes a result of Hakimi~\cite{Hakimi-1965} who gave a criterion for the existence of orientations with given upper bound on out-degrees.
\begin{thm}{\rm (\cite{Frank-Gyarfas-1976})}\label{thm:Into:Frank-Gyrfas}
{Let $G$ be a graph and let $\p$ and $\q$ be two integer-valued function on $V(G)$ with $\p\le \q$.
Then $G$ has a $(p,q)$-orientation if and only if for all $S\subseteq V(G)$,
$$e_G(S)\le \min\{\sum_{v\in S}q(v),\, \sum_{v\in S}(d_G(v)-\p(v))\}.$$ 
}\end{thm}

Recently, the present author introduced the following criterion for the existence of $\{\p, \q\}$-orientations in highly edge-connected graphs. In this note, we prove that under some conditions, as mentioned in the abstract, a simple graph has a $\{\p,p+1,q-1,\q\}$-orientation if and only if it has a $(p,q)$-orientation.
\begin{thm}{\rm (\cite{p-q})}\label{thm:Into:g,f-orientation}
{Let $G$ be a $8k^2$-edge-connected graph and let $p$ and $q$ be two integer-valued functions on $V(G)$ in which for each vertex $v$, $p(v)\le d_G(v)/2\le q(v)$ and $|q(v)-p(v)|\le k$.
Then $G$ has a $\{p,q\}$-orientation if and only if there is an integer-valued function $t$ on $V(G)$ such that $|E(G)|=\sum_{v\in V(G)}t(v)$ and $t (v)\in \{p(v),q(v)\}$ for each $v\in V(G)$.
}\end{thm}

As an application, we refine the following result in bipartite simple graphs which is due to Addario-Berry, Dalal, and Reed (2008). More precisely, we conclude that under simpler conditions, a bipartite simple graph has a $\{g,g+1,f-1,f\}$-factor if and only if it has a $(g,f)$-factor.
\begin{thm}{\rm (\cite{AddarioBerry-Dalal-Reed-2008})}\label{intro:thm:2008}
{Let $G$ be a simple graph and let $g$ and $f$ be two integer-valued functions on $V(G)$ satisfying $f\le d_G$.
If for each $v\in V(G)$, $\frac{1}{2} f(v)-2\le g(v)\le \lfloor \frac{1}{2}d_G(v)\rfloor< f(v)\le \frac{1}{2}(d_G(v)+g(v))+2$,
then $G$ has a factor $F$ such that for each $v\in V(G)$, $$d_F(v)\in \{g(v),g(v)+1, f(v)-1,f(v)\}.$$
}\end{thm}
\section{$\{\p,\p+1,q-1, \q\}$-orientations of simple graphs}
\label{sec:p,p+1,q-1,q}
The following theorem gives a criterion for existence of $\{\p,\p+1, \q-1,\q\}$-orientations in simple graphs.
\begin{thm}\label{thm:simplegraph}
{Let $G$ be a simple graph and let $\p$ and $\q$ be two integer-valued functions on $V(G)$ with $p<q$ in which
 for each $v\in V(G)$, $q(v) \ge \frac{1}{2}d_G(v)$ and $\p(v) \ge \frac{1}{2} \q(v)-2$. 
Then $G$ admits a $(p,q)$-orientation if and only if it has an orientation such that for each $v\in V(G)$, 
$$d^+_G(v)\in\{\p(v),p(v)+1,q(v)-1,\q(v)\}.$$
}\end{thm}
\begin{proof}
{Consider an orientation for $G$ such that for each vertex $v$, 
$\p(v) \le d^+_D(v)\le \q(v)$.
Now, among such orientations, consider $D$ with the minimum $\sum_{v\in W_D}|d^+_D(v)-\p(v)|$,
where $X(D)=\{v\in V(G):\p(v)+1<d^+_D(v)< \q(v)-1\}$.
If $X(D)=\emptyset$, then the proof is completed.
Suppose, to the contrary, that there is a vertex $x\in X(D)$.
Define $S$ to be the set of all vertices $v$ such that there is a directed path from $x$ to $v$. 
Note that we must have $d^+_D(v)\not \in\{ p(v),q(v)-1\}$;
otherwise, we can reverse the orientation of that path to obtain a better orientation, which derives a contradiction.
Obviously, $x\in S$.
By the definition of $S$, there is no directed edge from $S$ to $V(G)\setminus S$.
Therefore,
$$\sum_{v\in S}d^-_{D[S]}(v)= \sum_{v\in S}d^+_D(v).$$
This implies that $d^-_{D[S]}(x)\ge d^+_D(x)$ or
 there is a vertex $y\in S\setminus \{x\}$ such that $d^-_{D[S]}(y)\ge d^+_D(y)+1$. 
In the first case, 
since $d^+_D(x)\ge p(x)+2\ge \frac{1}{2}q(x)$, 
we must have 
$d^-_{D[S]}(x)\ge d^+_D(x)\ge q(x)-1-d^+_D(x)\ge 1$.
Thus
we can reverse the orientation of
$\q(x)-1-d^+_{D}(x)$ edges of $D[S]$ incident to $x$ which is directed toward it. 
In the second case, 
since $d^+_D(y)\ge p(y)+1\ge \frac{1}{2}q(y)-1$, 
similarly we must have 
$d^-_{D[S]}(y)\ge d^+_D(y)+1 \ge q(y)-1-d^+_D(y)$.
In addition, the inequality $q(y)\ge d_G(y)/2$ implies that $d^+_D(y)\neq q(y)$ and hence $q(y) -1- d^+_D(y) \ge 1 $.
Therefore, we can first reverse the orientation of a directed path from $x$ to $y$, 
and next reverse the orientation of 
$\q(y)-2-d^+_{D}(y)$ 
edges incident to $y$ which is directed toward it. 

Let $D_0$ be the new orientation of $G$ and let $X(D_0)=\{v\in V(G):\p(v)+1<d^+_{D_0}(v)< \q(v)-1\}$.
Since $G$ has no multiple edges, each $v\in S\setminus \{x, y\}$ is incident to at most one modified edge of the last step. 
This implies that $d^+_{D_0}(v)-d^+_D(v)\in \{-1,0\}$ and $p(v)\le d^+_{D_0}(v)\le q(v)$. 
Recall that $d^+_D(v)\not \in\{ p(v),q(v)-1\}$ for all $v\in S$.
For the first case, we have $d^+_{D_0}(x)=q(x)-1$ and hence $X(D_0)\subseteq X(D)\setminus \{x\}$.
For the second case, we have 
 $d^+_{D_0}(x)-d^+_D(x)\in \{-2,-1\}$ and $d^+_{D_0}(y)=q(y)-1$, and hence $X(D_0)\subseteq X(D)\setminus \{y\}$.
Therefore, $D_0$ is a $(p,q)$-orientation of $G$ while 
$\sum_{v\in X(D_0)}|d^+_{D_0}(v)-\p(v)|<\sum_{v\in X(D)}|d^+_D(v)-\p(v)|$.
This is a contradiction and consequently the theorem is proved
}\end{proof}
\begin{cor}\label{cor:simplegraph}
{Let $G$ be a simple graph and let $\p$ and $\q$ be two integer-valued functions on $V(G)$ with $p<q$.
 If for each $v\in V(G)$,
$\frac{1}{3}d_G(v)-\frac{4}{3}\le p(v)\le \frac{1}{2}d_G(v) \le q(v)\le\frac{2}{3}d_G(v)+\frac{4}{3}$,
then $G$ admits an orientation such that for each $v\in V(G)$, 
$$d^+_G(v)\in\{\p(v),p(v)+1,q(v)-1,\q(v)\}.$$
}\end{cor}
\begin{proof}
{Obviously, the graph $G$ has an orientation such that for each vertex $v$, $|d^+_G(v)-d^-_G(v)|\le 1$ which implies that
$q(v)\le \lfloor \frac{1}{2}d_G(v)\rfloor \le d^+_G(v)\le \lceil \frac{1}{2}d_G(v)\rceil\le p(v)$. 
Since $p(v)\ge \frac{1}{3}d_G(v)-\frac{4}{3}\ge \frac{1}{2}q(v)-2$, the proof can be completed by Theorem~\ref{thm:simplegraph}.
}\end{proof}
\begin{remark}
{Theorem~\ref{thm:simplegraph} can be reformulated by replacing the conditions 
 $q(v) \le (d_G(v)+\q(v))/2+2$ and $p(v) \le d_G(v)/2$. To see this, it is enough to work with restricted in-degrees in the proof.
By the same arguments in the proof, one can also develop Theorem~\ref{thm:simplegraph} to multigraphs $G$ provided that for each vertex $v$, $p(v)-q(v)/2+2\ge d_G(v)-|N_G(v)|$, where $N_G(v)$ denotes the set of all neighbours of $v$ in $G$. 
}\end{remark}
\section{Applications to degree constrained factors}
Addario-Berry, Dalal, McDiarmid, Reed, and Thomason (2007)
 established the following theorem on the existence of degree constrained factors in simple graphs. 
This result was a prototype of Theorem~\ref{intro:thm:2008}.
In this section, we are going to introduce a new stronger version for both of Theorems~\ref{intro:thm:2008} and~\ref{thm:Intro:2007} in bipartite simple graphs based Theorem~\ref{thm:simplegraph}.

\begin{thm}{\rm (\cite{AddarioBerry-Dalal-McDiarmid-Reed-Thomason-2007})}\label{thm:Intro:2007}
{Let $G$ be a simple graph and let $g$ and $f$ be two integer-valued functions on $V(G)$.
If for each $v\in V(G)$, $\frac{1}{3}d_G(v)-\frac{4}{3}\le g(v)\le \frac{1}{2}d_G(v) \le f(v)\le\frac{2}{3}d_G(v)+\frac{4}{3}$,
then $G$ has a factor $F$ such that for each $v\in V(G)$, $$d_F(v)\in \{g(v),g(v)+1, f(v)-1,f(v)\}.$$
}\end{thm}
For this purpose, we need the following lemma that provides a useful relation between orientation and factors of bipartite graphs.
A special case of this lemma was also used by Thomassen (2014)~\cite{Thomassen-2014} to form a result on modulo factors of edge-connected graphs.
\begin{lem}\label{lem:factor-orientation}
{Let $G$ be a bipartite graph with bipartition $(X,Y)$ and $L:V(G)\rightarrow 2^\mathbb{Z}$ be a function.
Then $G$ admits an $L$-orientation if and only if $G$ admits an $L_0$-factor, where for each vertex $v$,
$$L_0(v)=
 \begin{cases}
L(v),	&\text{when $v\in X$};\\
\{d_G(v)-i:i\in L(v)\},	&\text{when $v\in Y$}. 
\end {cases}$$
}\end{lem}
\begin{proof}
{If $D$ is an orientation of $G$, then the factor $F$ consisting of all edges of $G$ directed from $X$ to $Y$
satisfies $d_F(v)=d^+_D(v)$ for each $v\in X$, and $d_F(v)=d_G(v)-d^+_D(v)$ for each $v\in Y$.
Conversely, from every factor $F$, we can make an orientation $D$ whose edges directed from $X$ to $Y$ 
are exactly the same edges of $F$.
}\end{proof}
The following theorem is an equivalent version of Theorem~\ref{thm:simplegraph} in bipartite graphs in terms of factors.
Note that every bipartite graph $G$ has a factor $F$ such that for each vertex $v$, $\lfloor d_G(v)/2\rfloor \le d_F(v)\le \lceil d_G(v)/2\rceil$. To see this, it is enough to apply Lemma~\ref{lem:factor-orientation} along with an orientation of $G$ such that for each vertex $v$, $|d^+_G(v)-d^-_G(v)|\le 1$.
\begin{thm}\label{cor:g,g+1,f:bipartite}
{Let $G$ be a bipartite simple graph with bipartition $(X,Y)$ and let $g$ and $f$ be two integer-valued functions on $V(G)$ with $g<f$.
Assume that for each $v\in X$, $f(v)\ge d_G(v)/2$ and $g(v)\ge f(v)/2-2$,
and for each $v\in Y$, 
 $g(v)\le d_G(v)/2$ and $f(v)\le (d_G(v)+g(v))/2+2$.
Then $G$ has a $(g,f)$-factor if and only if it has a factor $F$ such that for each $v\in V(G)$,
$$d_F(v)\in \{g(v),g(v)+1, f(v)-1, f(v)\}.$$
}\end{thm}
\begin{proof}
{Apply Lemma~\ref{lem:factor-orientation} and Theorem~\ref{thm:simplegraph} 
by setting $p(v)=g(v)$ and $q(v)=f(v)$ for each $v\in X$, and setting $p(v)=d_G(v)-f(v)$ and $q(v)=d_G(v)-g(v)$ for each $v\in Y$.
}\end{proof}
\begin{remark}
{Note that one can use Lemma~\ref{lem:factor-orientation} to rediscover Theorem 4 and Lemma 6 in \cite{ADEOS},
which gave sufficient conditions for the existence of list orientations, from Theorem 1 in~\cite{Frank-Lau-Szabo-2008} and Theorem 2 in \cite{Frank-Lau-Szabo-2008}, which gave sufficient conditions for the existence of list factors.
}\end{remark}
%
%==============================================================================
%
%
%
%
%
%==============================================================================
%

%\bibliographystyle{siam}
%\bibliography{ref}
\end{document}